\documentclass[preprint,12pt,3p]{elsarticle}

\usepackage{amssymb}
\usepackage{amsthm}
\usepackage{amsmath}
\usepackage[hidelinks,colorlinks=false]{hyperref}
\newtheorem{theorem}{Theorem}
\newtheorem{lemma}[theorem]{Lemma}

\newtheorem{corollary}[theorem]{Corollary}
\newtheorem{conjecture}[theorem]{Conjecture}
\newtheorem{definition}[theorem]{Definition}
\newtheorem{proposition}[theorem]{Proposition}

\journal{Discrete Mathematics}

\begin{document}

\begin{frontmatter}

\title{A topological lower bound for the chromatic number\\
of a special family of graphs}

\author[label1]{Hamid Reza Daneshpajouh}
\address[label1]{School of Mathematics, Institute for Research in Fundamental Sciences (IPM),\\
Tehran, Iran, P.O. Box 19395-5746}


\ead{hr.daneshpajouh@gmail.com, hr.daneshpajouh@ipm.ir}



\begin{abstract}
For studying topological obstructions to graph colorings, Hom-complexes were introduced by Lov\'{a}sz.
A graph $T$ is called a test graph if for every
graph $H$, the $k$-connectedness of $|Hom(T, H)|$ implies $\chi (H)\geq k + 1 + \chi(T)$. The proof of the famous Kneser conjecture is based on the fact that $\mathcal{K}_2$, the complete graph on $2$ vertices, is a test graph. This result was extended to all complete graphs by Babson and Kozlov. 
Their proof is based on generalized nerve lemma 
and discrete Morse theory.
    
In this paper, we propose a new topological lower bound for the chromatic number of a special family of graphs. As an application of this bound, we give a new proof of the well-known fact that complete graphs and even cycles are test graphs.  
\end{abstract}


\end{frontmatter}


\section{Introduction}
Graph coloring is one of the most challengeable and practical topics in combinatorics. A proper (vertex) coloring is an assignment of colors to each vertex of a graph such that no edge connects two identically colored vertices. The smallest number of colors needed for proper coloring of a graph $G$ is the chromatic number, $\chi (G)$. In general, determining the chromatic number of a graph is an arduous problem. 
Even deciding whether a given planar graph is $3$-colorable is NP-complete problem~\cite{dailey1980uniqueness}.
In other words, it means that no convenient method is known for calculating the chromatic number of an arbitrary graph. This question now arises naturally: Can we at least make a "good" approximation on the number of colors we need? To estimate the chromatic number of a graph, we usually need to go through following steps:
\begin{itemize}
    \item Giving a proper coloring, to find out how many different colors are \textbf{sufficient} for coloring.
    \item Giving a rigorous argument, to show that how many different colors are \textbf{necessary} for coloring.
\end{itemize}
In other words, in the first part we obtain an upper bound and in the second part, we obtain a lower bound on the chromatic number. Whatever these bounds are closer, our estimation is better! 

To provide new (topological) lower bounds on the chromatic number of graphs, Hom complexes were defined by Lov\'{a}sz and they have been extensively studied by many authors, see ~\cite{kozlov2007combinatorial, dochtermann2009hom, schultz2009graph}. 
The breakthrough proof of the well-known Kneser conjecture by Lov\'{a}sz~\cite{lovasz1978kneser} implies the following lower bound on the chromatic number of graphs.
\begin{theorem}
If $|Hom(\mathcal{K}_2, H)|$ is $k$-connected, then $\chi (H)\geq k+3.$
\end{theorem}
This result was extended to all complete graphs by Babson and Kozlov, see~\cite[section 19.2]{kozlov2007combinatorial}.
\begin{theorem}
If $|Hom(\mathcal{K}_r, H)|$ is $k$-connected, then $\chi (H)\geq  k+r+1$.
\end{theorem}
To provide a better lower bound on the chromatic number of graphs, Lov\'{a}sz made the following conjecture:
\begin{conjecture}[Lov\'{a}sz]
Let $C_{2r+1}$ be the odd cycle with $2r+1$ vertices. If $|Hom(C_{2r+1}, H)|$ is $k$-connected, then $\chi (H)\geq  k+4$.
\end{conjecture}
More generally, Bj\"{o}rner and Lov\'{a}sz made the following conjecture to generalize the concept of a topological obstruction to graph coloring. 
\begin{conjecture}[Bj\"{o}rner and Lov\'{a}sz]
If $|Hom(T, H)|$ is $k$-connected, then $\chi (H)\geq  k+\chi (T) +1$.
\end{conjecture}
The first conjecture was confirmed by Babson and Kozlov~\cite{babson2007proof}, but the second one was disproved by Hoory and Linial~\cite{hoory2005counterexample}. Regarding Conjecture $2$, the following definition was proposed by Kozlov in~\cite{kozlov2007combinatorial}. A graph $T$ is called a (homotopy) test graph if for every graph $H$, the $k$-connectedness of $|Hom(T, H)|$ implies $\chi (H)\geq k + 1 + \chi(T)$. So in this terminology, complete graphs and odd cycles are test graphs. There are many more test graphs known besides these, see~\cite{schultz2009graph, schultz2011equivariant}.
 
In this paper, we propose a new topological lower bound for the chromatic number of a special family of graphs. 
Moreover, as an application of this bound, we give a new proof of the fact that complete graphs and even cycles are test graphs. One can interpret our approach as a combinatorial method to find test graphs. We hope that the proposed technique can be effectively used for finding new test graphs.

The organization of the paper is as follows. In Section $2$, we review some standard facts on simplicial complexes, partially ordered sets, and $G$-spaces. 
Finally, in Section $3$, our main results are stated and proved.

\section{Preliminaries and Notations}
In this paper, all graphs are finite, simple and undirected. Here and subsequently, $G$ stands for a non-trivial finite group, and its identity element is denoted by $e$. The following is a brief overview of some of the basic concepts we need.

\subsection{$G$-spaces and $G$-equivariant maps}
If $X$ is a set, then a group action of $G$ on $X$ is a function $G\times X\to X$ denoted $(g, x)\mapsto g\cdot x$, such that $e\cdot x = x$ and $(gh)\cdot x = g\cdot (h\cdot x)$ for all $g, h\in G$ and all $x\in X$. If $X$ is a topological space and $G$ is a topological group, then $X$ is called a $G$-space if $G$ acts continuously on $X$. If, moreover, the action is free, i.e, for all $x\in X$, $g\cdot x=x$ implies $g=e$, $X$ is called a free $G$-space. If $X$ and $Y$ are $G$-spaces, a continuous map $f: X\to Y$ is a $G$-equivariant map if $f(g\cdot x) = g\cdot f(x)$ for all $g\in G$ and all $x\in X$. 
\subsection{Simplicial complexes and $G$-simplicial complexes}
We assume that the reader is familiar with standard definitions and concepts of simplicial complexes. We just recall here the main definitions and notations used in the paper. For more background, see \cite{matousek2008using, wachs2006poset}. A (finite) simplicial complex $\mathcal{K}$ is a non-empty, hereditary set system of finite sets. That is, $F\in \mathcal{K}$, $F^{\prime}\subset F$ implies $F^{\prime}\in \mathcal{K}$ and we have $\emptyset\in \mathcal{K}$. The union of all elements of $\mathcal{K}$ is denoted by $V(\mathcal{K})$. The element of $V(\mathcal{K})$ are called vertices of $\mathcal{K}$, and the elements of $\mathcal{K}$ are called the simplicies of $\mathcal{K}$. The dimension of a simplex $\sigma\in \mathcal{K}$ is $\dim (\sigma)= |\sigma|-1$. The dimension of $\mathcal{K}$ is the maximum dimension of the simplices in $\mathcal{K}$. We denote the geometric realization of $\mathcal{K}$ by $|\mathcal{K}|$. A map $f : V(\mathcal{K})\to V(\mathcal{L})$ is called simplicial if it maps any simplex to a simplex, that is, $\sigma\in\mathcal{K}$ implies $f(\sigma)\in\mathcal{L}$. Every simplicial mapping $f:\mathcal{K}\to\mathcal{L}$ can be extended linearly to get a continuous mapping $|f| : |\mathcal{K}|\to |\mathcal{L}|$, which is called the affine extension of $f$. 

A simplicial $G$-complex is a simplicial complex together
with an action of $G$ on its vertices that takes simplices to simplices. Note that if $\mathcal{K}$ is a simplicial $G$-complex then the geometric realization $|\mathcal{K}|$ is a $G$-space under the natural induced action of $G$. Moreover, if the induced action of $G$ on $|\mathcal{K}|$ is free, then $\mathcal{K}$ is called a free simplicial $G$-complex. For two simplicial $G$-complexes $\mathcal{K}$ and $\mathcal{L}$, a simplicial map $f: V(\mathcal{K})\to V(\mathcal{L})$ is called a simplicial $G$-equivariant map if $f(g\cdot x)=g\cdot f(x)$ for all $g\in G$ and all $x\in V(\mathcal{K})$. One can easily see that, the affine extension of a simplicial $G$-equivariant map is a $G$-equivariant map.  
\subsection{Partially ordered sets and $G$-posets}
A partially ordered set or poset is a set and a binary relation $\leq$ such that for all $a, b, c\in P$:
$a\leq a$ (reflexivity); $a\leq b$ and $b\leq c$ implies $a\leq c$ (transitivity); and $a\leq b$ and $b\leq a$ implies $a = b$ (anti-symmetry). A pair of elements $a, b$ of a partially order set are called comparable if $a\leq b$ or $b\leq a$. A subset of a poset in which each two elements are comparable is called a chain. A function $f : P\to Q$ between partially ordered sets is order-preserving or monotone, if for all $a$ and $b$ in $P$, $a\leq_{P} b$ implies $f(a)\leq_{Q} f(b)$. The order complex of a poset $P$ is the simplicial complex $\Delta (P)$, whose vertices are the elements of $P$ and whose simplices are all chains in $P$.

A $G$-poset is a poset together with a $G$-action on its elements that preserves the partial order, i.e, $x<y \Rightarrow g\cdot x<g\cdot y$. A $G$-poset $P$ is called free $G$-poset, if for all $x$ in $X$, $g\cdot x=x$ implies $g=e$. One can see that, if $P$ is a free $G$-poset then its order complex $\Delta (P)$ is a free simplicial $G$-complex.
\subsection{Connectivity and $G$-index}
Let $k\geq -1$. A topological space $X$ is called $k$-connected if for every $-1\leq m\leq k$, each continuous mapping of the $m$-dimensional sphere $\mathbb{S}^m$ into $X$ can be extended to a continuous mapping of the $(m+1)$-dimensional ball $\mathbb{B}^{m+1}$. Here $\mathbb{S}^{-1}$ is interpreted as $\emptyset$ and $\mathbb{B}^0$ as a single point, and so $(-1)$-connected means nonempty. The largest $k$, if it exists, such that $X$ is $k$-connected is called the connectivity of $X$, denoted by $conn(X)$. If $X$ is $k$-connected for every $k\geq -1$, then we set $conn(X)=\infty$. A simplicial complex is called $k$-connected if its geometric realization is $k$-connected. For an integer $n\geq 0$ and a group $G$, an $\mathbb{E}_nG$ space is the geometric realization of an $(n-1)$-connected free $n$-dimensional simplicial $G$-complex. For a $G$-space $X$, we define
$$ind_{G} X = \min\{n |\,\text{there is a}\,G\text{-equivariant map}\, X\to \mathbb{E}_nG\}.$$
Note that the value of $ind_{G} X$ is independent of which $\mathbb{E}_nG$ space is chosen, since any of them $G$-equivariantly maps into any other, see~\cite[section 6.2]{matousek2008using} for details. In the following, we introduce a concrete example of an $\mathbb{E}_nG$ space that we will use in this paper. Let $G\times \{1,\cdots , n+1\}$ be the $G$-poset with its natural $G$-action, $h\cdot (g, i)\to (hg, i)$, and the order defined by $(h, x) < (g, y)$ if $x < y$. One can see that the geometric realization of order complex of this $G$-poset, $|\Delta (G\times\{1, 2,\cdots ,n+1\})|$, is an example of $\mathbb{E}_nG$ space. Let us finish this section by listing some basic properties of $ind_{G} X$. 

\begin{proposition}[\cite{matousek2008using}]
Let $G$ be a finite nontrivial group, and let $X, Y$ be $G$-spaces.
\begin{enumerate}
    \item If there is $G$-map from $X$ to $Y$, we have $ind_{G} X\leq ind_{G} Y$.
    \item For any $\mathbb{E}_nG$ spaces, $ind_{G} \mathbb{E}_nG = n$.
    \item $conn(X)+1\leq ind_{G} X$.
    \item If $\mathcal{K}$ is a free $G$-simplicial complex of dimension $n$, then $ind_G |\mathcal{K}|\leq n$.
\end{enumerate}
\end{proposition}

\section{Compatibility graphs and Test graphs}
We begin this section by introducing compatibility graphs.
\begin{definition}[Compatibility graph]
Let $P$ be a $G$-poset. The compatibility graph of $P$, denoted by $C_P$, is the graph $C_P$ with vertex set $P$, and two elements $x, y\in P$ are adjacent if there is an element $g\in G\setminus\{e\}$ such that $x$ and $g\cdot y$ are comparable in $P$.
\end{definition}
In general, it needs to be noted that $C_p$ may have loops, and therefore it is not colorable! So,
we need to limit ourselves to those $G$-posets whose corresponding graphs do not contain any loops. In the following lemma, it can be seen that if $P$ is "nice" enough, then $C_P$ contains no loops.
\begin{lemma}
If $P$ is a free $G$-poset, then $C_P$ contains no loops.
\end{lemma}
\begin{proof}
Suppose, the contrary to our claim, that $C_P$ contains a loop. Let $x$ be a vertex of $C_P$ which is connected to itself. By the definition of compatibility graph, there is a $e\neq g\in G$ such that $x$ and $g\cdot x$ are comparable in $P$. Since $P$ is free, $x\neq g\cdot x$. Without loss of generality suppose that $x < g\cdot x$. If we multiply the both sides of the previous inequality by $e, g, \cdots , g^{|G|-1}$, respectively, we will find
\begin{align*}
x &< g\cdot x\\
g\cdot x &< g^2\cdot x\\
&\vdots\\
g^{|G|-1}\cdot x &< g^{|G|}\cdot x=x.
\end{align*}
Now, by the transitivity of $<$, $x < x$, which is impossible.
\end{proof}
Now we are in a position to state the main result of the paper.
\begin{theorem}
If $P$ is a finite free $G$-poset, 
then
$$ind_{G} |\Delta (P)|+ |G|\leq\chi\left(C_P\right).$$
\end{theorem}
\begin{proof}
By Lemma $7$ the compatibility graph $C_P$ contains no loops, and thus it is colorable.
Let $c : C_P\to\{1, \cdots , C\}$ be a proper coloring of $C_P$ with $C$ colors. This coloring induces a simplicial $G$-equivariant map as follows.
\begin{align*}
\lambda : \Delta (P) &\to \Delta (G\times\{1,\cdots , C-|G|+1\})\\
&x\longmapsto (\lambda_1(x), \lambda_2(x)).
\end{align*} 
First note that $C-|G|+1\geq 1$, since for each $x\in P$ the vertices in $\{g\cdot x| g\in G\}$ form a clique of size $|G|$ in $C_P$, therefore $C\geq |G|$. For a given $x$ in $P$, let $g_x\cdot x$ be the one among all $g\cdot x$, $g\in G$ with the minimum color value in $C_P$. That is, $c(g_x\cdot x) = \min\{c\left(g\cdot x\right) : g\in G\}$. Note that since every two distinct vertices in $\{g\cdot x : g\in G\}$ are adjacent, the element $g_x\cdot x$ is uniquely determined. Now, set
$\lambda (x)= \left(g_x^{-1}, c\left(g_x\cdot x\right)\right).$
To show that $\lambda$ is a simplicial $G$-equivariant map, we need to check:
\begin{enumerate}
    \item For each $g\in G$ and $x\in P$, $\lambda(g\cdot x)=g\cdot \lambda(x)$.
    \item 
    We have to show that $\lambda$ is a simplicial map, i.e., takes any simplex $\sigma = \{p_1 < \cdots < p_m\}$ in $\Delta (P)$ to a simplex $\lambda (\sigma) = \{\lambda(p_1) = (\lambda_1(p_1), \lambda_2(p_1)), \cdots, \lambda(p_m) = (\lambda_1(p_m), \lambda_2(p_m))\}$ in $\Delta(G\times\{1, \cdots ,C-|G|+1\})$. Note that, according to the definition of $\Delta(G\times\{1, \cdots ,C-|G|+1\})$, $\lambda (\sigma)$ is a simplex if and only if it contains no two different elements with the same second entries. Thus, $\lambda$ is a simplicial map if and only if for all comparable elements $x, y$ in $P$, if $\lambda_2 (x)=\lambda_2(y)$, then $\lambda_1(x)=\lambda_1(y)$.
\end{enumerate}
Let $x\in P$ and $g\in G$. By the definition of $\lambda$, there exists a $g_x\in G$ such that $\lambda(x) = (g_x^{-1}, c(g_x\cdot x))$. Since $g_x\cdot x= g_xg^{-1}(g\cdot x)$ and $g_x\cdot x$ is uniquely determined in $C_P$, therefore 
$$\lambda(g\cdot x) = \left({\left(g_xg^{-1}\right)}^{-1} , c(g_xg^{-1}(g\cdot x))\right) = \left(gg_x^{-1} , c(g_x\cdot x)\right)=g\cdot \lambda(x).$$
Now, let $x$ and $y$ be comparable elements of $P$. Also, let $\lambda(x) = (g_x^{-1}, c(g_x\cdot x))$, $\lambda(y) = (g_y^{-1}, c(g_y\cdot y))$, and $c(g_x\cdot x) = c(g_y\cdot y)$. To finish the proof, we need to show that $g_x=g_y$. Suppose, contrary to our claim, that is $g_x\neq g_y$. 
Since $x$ and $y$ are comparable, then $g_x\cdot x$ and $g_x\cdot y$ are also comparable. Moreover $g_x\cdot y = (g_xg_y^{-1})g_y\cdot y$, therefore $g_x\cdot x$ and $g_y\cdot y$ are adjacent as $g_xg_y^{-1}\neq e$. This contradicts the fact that $c$ is a proper coloring of $C_P$.

Therefore, $\lambda$ is a $G$-simplicial map. Its map naturally lifts to the $G$-equivariant map between $G$-spaces $|\Delta (P)|$ and $|\Delta\left (G\times\{1,\cdots , C-|G|+1\}\right)|$. Thus, according to Proposition $5$, $ind_{G} |\Delta (P)|\leq ind_{G} |\Delta\left (G\times\{1,\cdots , C-|G|+1\}\right)| = C-|G|$. Therefore,
$$ind_{G} |\Delta (P)|+ |G|\leq\chi (C_P).$$
\end{proof}
Now, let us mention two interesting consequences of the theorem. First, let us recall the definition of the Hom complex. As we pointed out, this concept was defined by Lov\'{a}sz and it has been studied by many authors to provide topological lower bounds on the chromatic number of graphs. We need the following version of this concept.
\begin{definition}[Hom poset and Hom complex]
Let $F$ be a graph with vertex set $\{1, 2\cdots , n\}$. For a graph $H$, we define $Hom_{p}(F, H)$, Hom poset, to be the poset whose elements are given by all $n$-tuples $(A_1, \cdots, A_n)$ of non-empty subsets of $V(H)$, such that for any edge $(i, j)$ of $F$ we have $A_i\times A_j\subseteq E(H)$. The partial order is defined by $A=(A_1,\cdots, A_n)\leq B=(B_1,\cdots, B_n)$ if and only if $A_i\subseteq B_i$ for all $i\in\{1,\cdots , n\}$. We define the Hom complex $Hom(F, H)$ as the order
complex of the resulting poset.
\end{definition}
Let $\mathbb{Z}_r=\{e=\omega^0,\cdots, \omega^{r-1}\}$ be the cyclic group of order $r$, and let $\mathcal{K}_r$ be the complete graph. The cyclic group $\mathbb{Z}_r$ acts on the poset $Hom_{p}(\mathcal{K}_r, H)$ naturally by cyclic shift. More precisely, for each $\omega^{i}\in \mathbb{Z}_r$ and $(A_1,\cdots, A_r)\in Hom_{p}(\mathcal{K}_r, H)$ , define $\omega^{i}\cdot (A_1,\cdots, A_r)=(A_{1+i(\text{mod} r)},\cdots , A_{r+i(\text{mod} r)})$. Clearly this action is free on $Hom_{p}(\mathcal{K}_r, H)$, and consequently $Hom_{p}(\mathcal{K}_r, H)$ is a free $\mathbb{Z}_{r}$-poset. 
Therefore, by Theorem $8$, one the one hand we have 
$$r+ind_{G} |Hom(\mathcal{K}_r,H)|\leq \chi\left(C_{Hom_{p}(\mathcal{K}_r, H)}\right).$$
On the other hand, we have an obvious graph homomorphism from $C_{Hom_{p}(\mathcal{K}_r, H)}$ to $H$ by sending each vertex $(A_1, \cdots, A_r)$ to an arbitrary element of $A_1$. This give us the following lower bound on chromatic number
$$\chi (C_{Hom_{p}(\mathcal{K}_r, H)})\leq\chi (H).$$
The following corollary is an immediate consequence of the definition of the test graphs, Proposition $5.3$, and the above argument.
\begin{corollary}
All complete graphs are test graphs.
\end{corollary}
As another application of our main result, we also reprove that all cycles of even length are test graphs. Let $C_{2r}$ be a cycle of even length with vertex set $\{1,\cdots, 2r\}$. The group $\mathbb{Z}_2=\{e=\omega^0, \omega\}$ acts on $Hom_{p}(C_{2r}, H)$ as follows, $\omega\cdot (A_1,\cdots, A_{i}, \cdots,A_{2r})=(A_{2r},A_{2r-1},\cdots , A_{2}, A_{1})$. Clearly, $Hom_{p}(C_{2r}, H)$ with this action is a free $\mathbb{Z}_2$-poset. Now, by Theorem $8$, we have
$$2+ind_{G} |Hom(C_{2r},H)|\leq\chi\left(C_{Hom_{p}(C_{2r}, H)}\right).$$
Indeed, we have a graph homomorphism from $C_{Hom(C_{2r}, H)}$ to $H$ by sending each vertex $(A_1, \cdots, A_{2r})$ to an arbitrary element of $A_1$. Thus,
$$\chi (C_{Hom(C_{2r}, H)})\leq\chi (H).$$
In conclusion, we have the following corollary.
\begin{corollary}
All cycle graphs of even length are test graphs.
\end{corollary}
\section*{Acknowledgements}
This paper is a part of my PhD thesis under supervision of Professor Hossein Hajiabolhassan. I wish to express my most sincere gratitude to him, for his guidance, great support and kind advice throughout my PhD research studies. 





\begin{thebibliography}{10}
\bibitem{babson2007proof}
 Babson, Eric, and Dmitry N. Kozlov. "Proof of the Lov\'{a}sz conjecture." Annals of Mathematics (2007): 965-1007.
\bibitem{dailey1980uniqueness}
Dailey, David P. "Uniqueness of colorability and colorability of planar 4-regular graphs are NP-complete." Discrete Mathematics 30.3 (1980): 289-293.
\bibitem{dochtermann2009hom}
Dochtermann, Anton. "Hom complexes and homotopy theory in the category of graphs." European Journal of Combinatorics 30.2 (2009): 490-509.
\bibitem{hoory2005counterexample}
Hoory, Shlomo, and Nathan Linial. "A counterexample to a conjecture of Bj\"{o}rner and Lov\'asz on the $\ chi $-coloring complex." arXiv preprint math/0405339 (2004).
\bibitem{kozlov2007combinatorial}
Kozlov, Dimitry. Combinatorial algebraic topology. Vol. 21. Springer Science \& Business Media, 2007.
\bibitem{lovasz1978kneser}
Lov\'{a}sz, L\'{a}szl\'{o}. "Kneser's conjecture, chromatic number, and homotopy." Journal of Combinatorial Theory, Series A 25.3 (1978): 319-324.
\bibitem{matousek2008using}
Matousek, Jiri. Using the Borsuk-Ulam theorem: lectures on topological methods in combinatorics and geometry. Springer Science \& Business Media, 2008.
\bibitem{matouvsek2004combinatorial}
Matou\v{s}ek, Ji\v{r}\'{i}. "A combinatorial proof of Kneser’s conjecture." Combinatorica 24.1 (2004): 163-170.
\bibitem{palvolgyi2009combinatorial}
P{\'a}lv{\"o}lgyi, D{\"o}m{\"o}t{\"o}r. "Combinatorial necklace splitting." the electronic journal of combinatorics 16.1 (2009): R79.
\bibitem{schultz2009graph}
Schultz, Carsten. "Graph colorings, spaces of edges and spaces of circuits." Advances in Mathematics 221.6 (2009): 1733-1756.
\bibitem{schultz2011equivariant}
Schultz, Carsten. "The equivariant topology of stable Kneser graphs." Journal of Combinatorial Theory, Series A 118.8 (2011): 2291-2318.
\bibitem{steinlein1985borsuk}
Steinlein, Heinrich. "Borsuk’s antipodal theorem and its generalizations and applications: a survey." Topological Methods in Nonlinear Analysis (A. Granas, ed.), S\'{e}m Mathematics Sup 95 (1985): 166-235.
\bibitem{tucker1945some}
Tucker, Albert William. Some Topological Properties of Disk and Sphere... 1945.
\bibitem{wachs2006poset}
Wachs, Michelle L. "Poset topology: tools and applications." arXiv preprint math/0602226 (2006).


\end{thebibliography}

\end{document}